\documentclass[11pt,twoside,a4paper]{article}
\usepackage{color,amscd,amsmath,amssymb,amsthm,latexsym,stmaryrd,authblk,mathabx,shuffle,hyperref,tikz,tkz-tab} 
\usepackage[latin1]{inputenc}
\usepackage[T1]{fontenc}   
\usepackage[english]{babel}
\providecommand{\keywords}[1]{\textbf{\textit{Keywords.}} #1}
\providecommand{\AMSclass}[1]{\textbf{\textit{AMS classification.}} #1}

\newcommand{\tun}{\begin{tikzpicture}[line cap=round,line join=round,>=triangle 45,x=0.5cm,y=0.5cm]
\clip(-0.2,-0.1) rectangle (0.2,0.2);
\begin{scriptsize}
\draw [fill=black] (0.,0.) circle (1pt);
\end{scriptsize}
\end{tikzpicture}}

\newcommand{\tdeux}{\begin{tikzpicture}[line cap=round,line join=round,>=triangle 45,x=0.5cm,y=0.5cm]
\clip(-.2,-.1) rectangle (0.2,0.7);
\draw [line width=.5pt] (0.,0.5)-- (0.,0.);
\begin{scriptsize}
\draw [fill=black] (0.,0.) circle (1pt);
\draw [fill=black] (0.,0.5) circle (1pt);
\end{scriptsize}
\end{tikzpicture}}

\newcommand{\ttroisun}{\begin{tikzpicture}[line cap=round,line join=round,>=triangle 45,x=0.5cm,y=0.5cm]
\clip(-0.5,-0.1) rectangle (0.5,0.7);
\draw [line width=0.5pt] (0.,0.)-- (-0.3,0.5);
\draw [line width=0.5pt] (0.,0.)-- (0.3,0.5);
\begin{scriptsize}
\draw [fill=black] (-0.3,0.5) circle (1pt);
\draw [fill=black] (0.,0.) circle (1pt);
\draw [fill=black] (0.3,0.5) circle (1pt);
\end{scriptsize}
\end{tikzpicture}}
\newcommand{\ttroisdeux}{\begin{tikzpicture}[line cap=round,line join=round,>=triangle 45,x=0.5cm,y=0.5cm]
\clip(-.2,-.1) rectangle (0.2,1.2);
\draw [line width=0.5pt] (0.,0.5)-- (0.,0.);
\draw [line width=0.5pt] (0.,0.5)-- (0.,1.);
\begin{scriptsize}
\draw [fill=black] (0.,0.) circle (1pt);
\draw [fill=black] (0.,0.5) circle (1pt);
\draw [fill=black] (0.,1.) circle (1pt);
\end{scriptsize}
\end{tikzpicture}}

\newcommand{\tquatreun}{\begin{tikzpicture}[line cap=round,line join=round,>=triangle 45,x=0.5cm,y=0.5cm]
\clip(-0.5,-0.1) rectangle (0.5,0.7);
\draw [line width=0.5pt] (0.,0.)-- (-0.3,0.5);
\draw [line width=0.5pt] (0.,0.)-- (0.3,0.5);
\draw [line width=0.5pt] (0.,0.)-- (0.,0.5);
\begin{scriptsize}
\draw [fill=black] (-0.3,0.5) circle (1.0pt);
\draw [fill=black] (0.,0.) circle (1.0pt);
\draw [fill=black] (0.3,0.5) circle (1.0pt);
\draw [fill=black] (0.,0.5) circle (1.0pt);
\end{scriptsize}
\end{tikzpicture}}
\newcommand{\tquatredeux}{\begin{tikzpicture}[line cap=round,line join=round,>=triangle 45,x=0.5cm,y=0.5cm]
\clip(-0.5,-0.1) rectangle (0.5,1.2);
\draw [line width=0.5pt] (0.,0.)-- (-0.3,0.5);
\draw [line width=0.5pt] (0.,0.)-- (0.3,0.5);
\draw [line width=0.5pt] (-0.3,0.5)-- (-0.3,1.);
\begin{scriptsize}
\draw [fill=black] (-0.3,0.5) circle (1.0pt);
\draw [fill=black] (0.,0.) circle (1.0pt);
\draw [fill=black] (0.3,0.5) circle (1.0pt);
\draw [fill=black] (-0.3,1.) circle (1.0pt);
\end{scriptsize}
\end{tikzpicture}}
\newcommand{\tquatretrois}{\begin{tikzpicture}[line cap=round,line join=round,>=triangle 45,x=0.5cm,y=0.5cm]
\clip(-0.5,-0.1) rectangle (0.5,1.2);
\draw [line width=0.5pt] (0.,0.)-- (-0.3,0.5);
\draw [line width=0.5pt] (0.,0.)-- (0.3,0.5);
\draw [line width=0.5pt] (0.3,0.5)-- (0.3,1.);
\begin{scriptsize}
\draw [fill=black] (-0.3,0.5) circle (1.0pt);
\draw [fill=black] (0.,0.) circle (1.0pt);
\draw [fill=black] (0.3,0.5) circle (1.0pt);
\draw [fill=black] (0.3,1.) circle (1.0pt);
\end{scriptsize}
\end{tikzpicture}}
\newcommand{\tquatrequatre}{\begin{tikzpicture}[line cap=round,line join=round,>=triangle 45,x=0.5cm,y=0.5cm]
\clip(-0.5,-0.1) rectangle (0.5,1.7);
\draw [line width=0.5pt] (0.,0.)-- (0.,0.5);
\draw [line width=0.5pt] (0.,0.5)-- (0.3,1.);
\draw [line width=0.5pt] (0.,0.5)-- (-0.3,1.);
\begin{scriptsize}
\draw [fill=black] (0.,0.) circle (1.0pt);
\draw [fill=black] (0.,0.5) circle (1.0pt);
\draw [fill=black] (-0.3,1.) circle (1.0pt);
\draw [fill=black] (0.3,1.) circle (1.0pt);
\end{scriptsize}
\end{tikzpicture}}
\newcommand{\tquatrecinq}{\begin{tikzpicture}[line cap=round,line join=round,>=triangle 45,x=0.5cm,y=0.5cm]
\clip(-.2,-.1) rectangle (0.5,1.7);
\draw [line width=0.5pt] (0.,0.)-- (0.,0.5);
\draw [line width=0.5pt] (0.,0.5)-- (0.,1.);
\draw [line width=0.5pt] (0.,1.)-- (0.,1.5);
\begin{scriptsize}
\draw [fill=black] (0.,0.) circle (1.0pt);
\draw [fill=black] (0.,0.5) circle (1.0pt);
\draw [fill=black] (0.,1.) circle (1.0pt);
\draw [fill=black] (0.,1.5) circle (1.0pt);
\end{scriptsize}
\end{tikzpicture}}

\input{xy}
\xyoption{all}

\title{The antipode of of a Com-PreLie Hopf algebra}
\date{}
\author{Lo\"\i c Foissy}
\affil{\small{Univ. Littoral Côte d'Opale, UR 2597
LMPA, Laboratoire de Mathématiques Pures et Appliquées Joseph Liouville
F-62100 Calais, France}.\\ Email: \texttt{foissy@univ-littoral.fr}}

\theoremstyle{plain}
\newtheorem{theo}{Theorem}[section]
\newtheorem{lemma}[theo]{Lemma}

\newtheorem{prop}[theo]{Proposition}
\newtheorem{defi}[theo]{Definition}

\theoremstyle{remark}
\newtheorem{remark}{Remark}[section]
\newtheorem{notation}{Notations}[section]
\newtheorem{example}{Example}[section]

\newcommand{\K}{\mathbb{K}}
\newcommand{\N}{\mathbb{N}}

\newcommand{\id}{\mathrm{Id}}
\newcommand{\tdelta}{\tilde{\Delta}}
\newcommand{\F}{\mathbb{F}}
\renewcommand{\ker}{\mathrm{Ker}}

\begin{document}

\maketitle

\begin{abstract}
We study the compatibility between the antipode and the preLie product of a Com-PreLie Hopf algebra, that is to say a commutative Hopf algebra with a complementary preLie product, 
compatible with the product and the coproduct in a certain sense. An example of such a Hopf algebra is the Connes-Kreimer Hopf algebra, with the preLie product given by grafting
of forests, extending the free preLie product of grafting of rooted trees. This compatibility is then used to study the antipode of the Connes-Moscovici subalgebra, which
can be defined with the help of this preLie product. The antipode of the generators of this subalgebra gives a family of combinatorial coefficients indexed by partitions,
which can be computed with the help of iterated harmonic sums. 
\end{abstract}

\\

\AMSclass{16T05 16T30 05C05 16W25 17A30}

\tableofcontents

\section*{Introduction}

PreLie algebras, also called Vinberg algebras or symmetric algebras (see \cite{Cartier2010,Manchon2011} for surveys on these objects) are particular Lie algebras, where the Lie bracket
is obtained by the antisymmetrization of a non associative product $\bullet$ such that
\begin{align*}
&\forall a,b,c\in A,&(a\bullet b)\bullet c-a\bullet(b\bullet c)&=(a\bullet c)\bullet b-a\bullet(c\bullet b).
\end{align*}
An easy and classical exercise proves that if this axiom is satisfied, the antisymmetrization of $\bullet$ indeed satisfies the Jacobi identity. A classical example is given by the Lie algebra of derivations 
of the polynomial algebra $\K[X]$, with the preLie product defined by
\begin{align*}
&\forall P,Q\in \K[X],&P\frac{\mathrm{d}X}{\mathrm{d}X}\bullet Q \frac{\mathrm{d}X}{\mathrm{d}X}&=-\frac{\mathrm{d}P}{\mathrm{d}X}Q \frac{\mathrm{d}X}{\mathrm{d}X}.
\end{align*}
A remarkable property of these particular Lie algebras is that an explicit description of their enveloping algebra can be done, under the name of Guin-Oudom's construction \cite{Oudom2005}. 
They are closely related to the combinatorics of rooted trees \cite{Chapoton2001}, and the enveloping algebra of the free preLie algebra is the Grossman-Larson Hopf algebra \cite{Grossman89,Grossman90},
dual to the Connes-Kreimer Hopf algebra \cite{Connes98}.

In \cite{Mansuy2013}, a new type of preLie algebras, called Com-PreLie algebras are introduced, see Definition \ref{defi1.1}. 
They also appear in \cite{Foissy28,Foissy30}, as structures appearing on the underlying Lie algebra of a group of Fliess operators. 
Such an object is given a preLie product $\bullet$ and a commutative, associative product $\cdot$, with a Leibniz-style compatibility between them:
\begin{align*}
&\forall a,b,c\in A,&(a\cdot b)\bullet c&=(a\bullet c)\cdot b+a\cdot (b\bullet c).
\end{align*}
Examples on trees are given in \cite{Mansuy2013}, on words in \cite{Foissy30,Foissy28,Foissy55}, and other combinatorial examples can be found in \cite{Foissy56}. 
Among them, the Connes-Kreimer Hopf algebra $H_{CK}$ inherits a complementary Com-PreLie structure. Let us give a few more details on this object. 
A basis of $H_{CK}$ is given by the set of rooted forests: 
\begin{align*}
&\left\{\begin{array}{c}
1,\tun,\tun\tun,\tdeux,\tun\tun\tun,\tdeux\tun,\ttroisun,\ttroisdeux,\\
\tun\tun\tun\tun,\tdeux\tun\tun,\tdeux\tdeux,\ttroisun\tun,\ttroisdeux\tun,\tquatreun,\tquatredeux,\tquatrequatre,\tquatrecinq\ldots
\end{array}\right\}
\end{align*}
Its commutative associative product is given by the disjoint union, for example:
\begin{align*}
\tdeux \cdot \tun \ttroisun&=\tdeux\tun\ttroisun,&\tun\tun\cdot \tquatreun&=\tun\tun\tquatreun.
\end{align*}
Its preLie product is given by graftings, for example:
\begin{align*}
\tdeux\tun\bullet \tun&=\ttroisun\tun+\ttroisdeux\tun+\tdeux\tdeux,&\ttroisdeux\bullet \tun&=\tquatredeux+\tquatrequatre+\tquatrecinq,&\tdeux\bullet \tdeux&=\tquatredeux+\tquatrecinq.
\end{align*}
This preLie product includes two important operators on $H_{CK}$, the grafting $B$ and the growth $N$: 
\begin{align*}
&\forall a\in H_{CK},&B(a)&=\tun \bullet a,&N(a)&=a\bullet \tun.
\end{align*}
IAs a Hopf algebra, $H_{CK}$ has a coproduct $\Delta$, consisting of separating rooted forests into a upper part (on the right) and a lower part (on the left), for example:
\begin{align*}
\Delta(\tquatredeux)&=\tquatredeux\otimes 1+\ttroisun\otimes \tun+\ttroisdeux\otimes\tun+\tdeux\otimes \tdeux+\tdeux\otimes \tun\tun+\tun\otimes \tdeux\tun+1\otimes \tquatredeux,\\
\Delta(\tquatreun)&=\tquatreun\otimes 1+3\ttroisun\otimes \tun+3\tdeux\otimes \tun\tun+\tun\otimes \tun\tun\tun+1\otimes \tquatreun.
\end{align*}
This coproduct satisfies the following compatibility with the preLie product:
\begin{align*}
&\forall a,b\in H_{CK},&\Delta(a\bullet b)&=a^{(1)}\otimes a^{(2)}\bullet b+a^{(1)}\bullet b^{(1)}\otimes a^{(2)} b^{(2)},
\end{align*}
with Sweedler's notation. This includes the well-known compatibilities \cite{Connes98}:
\begin{align*}
&\forall a\in H_{CK},&\Delta\circ B(a)&=1\otimes B(a)+ (B\otimes \id)\circ \Delta(a),\\
&&\Delta\circ N(a)&=N(a^{(1)})\otimes a^{(2)}+a^{(1)}\bullet 1\otimes a^{(2)}\tun.
\end{align*}
Such a family $(A,m,\bullet,\Delta)$ is a Com-PreLie bialgebra, and if $(A,m,\Delta)$ is a Hopf algebra, a Com-PreLie Hopf algebra. \\

Our first aim here is a study of the compatibility between the antipode $S$ of a Com-PreLie bialgebra and its preLie product. Keeping in mind the Connes-Kreimer case, we restrict ourselves to the case
where the bialgebra $(A,m,\Delta)$ is conilpotent, and we prove that 
\begin{align*}
&\forall a,b\in A,&S(a\bullet b)&=\left(S(a)\bullet b^{(1)}\right)S(b^{(2)}).
\end{align*}

Secondly, we use this result to study the antipode of the Connes-Moscovici Hopf subalgebra $H_{CM}$ of $H_{CK}$ \cite{Connes98}. With the preceding notation, $H_{CM}$ is the subalgebra generated by the elements
\[\delta_n=\underbrace{(\ldots((\tun\bullet \tun)\bullet \tun)\ldots)\bullet \tun}_{\mbox{\scriptsize $\tun$ appears $n$ times}},\]
with $n\geq 1$. For example,
\begin{align*}
\delta_1&=\tun,&\delta_2&=\tdeux,&\delta_3&=\ttroisun+\ttroisdeux,&\delta_4&=\tquatreun+3\tquatredeux+\tquatrequatre+\tquatrecinq.
\end{align*}
This subalgebra is stable under the coproduct. For example,
\begin{align*}
\Delta(\delta_1)&=\tun\otimes 1+1\otimes \tun\\
&=\delta_1\otimes 1+1\otimes \delta_1,\\
\Delta(\delta_2)&=\tdeux \otimes 1+\tdeux \otimes 1+\tun\otimes \tun\\
&=\delta_2\otimes 1+1\otimes \delta_2+\delta_1\otimes \delta_1,\\
\Delta(\delta_3)&=(\ttroisun+\ttroisdeux)\otimes 1+1\otimes (\ttroisun+\ttroisdeux)
+3\tdeux\otimes \tun+\tun\otimes \tdeux+\tun\otimes \tun\tun\\
&=\delta_3\otimes 1+1\otimes \delta_3+3\delta_2\otimes \delta_1+\delta_1\otimes \delta_2+\delta_1\otimes \delta_1^2,\\
\Delta(\delta_4)&=(\tquatreun+3\tquatredeux+\tquatrequatre+\tquatrecinq)\otimes 1+1\otimes (\tquatreun+3\tquatredeux+\tquatrequatre+\tquatrecinq)\\
&+(6\ttroisun+6\ttroisdeux)\otimes \tun+\tdeux\otimes (4\tdeux +7\tun\tun)+\tun \otimes (\ttroisun+\ttroisdeux+3\tdeux\tun+\tun\tun\tun)\\
&=\delta_4\otimes 1+1\otimes \delta_4+6\delta_3\otimes \delta_1+
\delta_2\otimes (4\delta_2+7\delta_1^1)+\delta_1\otimes (\delta_3+3\delta_2\delta_1+\delta_1^3).
\end{align*}
Note that for any $n\geq 1$, $\Delta(\delta_n)-\delta_n\otimes 1-\delta_n\otimes 1$ makes only appear $\delta_1,\ldots,\delta_{n-1}$, which allows to inductively compute $S(\delta_n)$, as $m\circ(S\otimes \id)\circ \Delta(\delta_n)=0$ for any $n\geq 1$: for example, we obtain
\begin{align*}
S(\delta_1)&=-\delta_1,&
S(\delta_2)&=-\delta_2+\delta_1^2,&
S(\delta_3)&=-\delta_3+4\delta_1\delta_2-2\delta_1^3.
\end{align*}
More generally, one can write
\begin{align*}
S(\delta_n)&=\sum_{\substack{(i_1,\ldots,i_n) \in \N^n,\\1i_1+\ldots+ni_n=n}}a_{i_1,\ldots,i_n}\delta_1^{i_1}\ldots \delta_n^{i_n},
\end{align*}
where the coefficients $a_{i_1,\ldots,i_n}$ are integers. Our aim is to give a way to compute these coefficients.
They have been studied in \cite{Menous2011}, in a very different approach: in particular, they are split into a sum of several coefficients, corresponding to the different way to write each monomial. 
Firstly, using the compatibility between the preLie product $\bullet$ and the antipode, as for any $n\geq 1$, $\delta_{n+1}=\delta_n\bullet \tun$, we obtain the inductive formula
\[a_{i_1,\ldots,i_n}=\sum_{j\geq 2,\: i_j \geq 1}(i_{j-1}+1)a_{i_1,\ldots,i_{j-1}+1,i_j-1,\ldots,i_{n-1}}
-(n-1)\delta_{i_1\geq 1}a_{i_1-1,i_2,\ldots,i_{n-1}}.\]
We then partially solve this induction. We prove the existence of a sequence of polynomials $P_{i_2,\ldots,i_n} \in \mathbb{Z}[X_1,\ldots,X_N]$, indexed by sequences of integers $(i_2,\ldots,i_k)$ with $i_k\neq 0$, with
\[N=1i_2+\ldots+(k-1)i_k-1,\]
such that for any $i_1\geq 0$,
\[a_{i_1,\ldots,i_k,0,\ldots,0}=(-1)^{i_1+\ldots+i_k}(1i_1+\ldots+ki_k-1)!\sum_{1\leq p_1<\ldots<p_N\leq 1i_1+\ldots+ki_k-1}
\frac{P_{i_2,\ldots,i_k}(p_1,\ldots,p_N)}{p_1\ldots p_N}.\]
These polynomials are inductively defined and can be rather complicated. For example,
\begin{align*}
P_{0,1}&=X_1-1,\\
P_2&=(X_1-2)(X_1-1),\\
P_{1,1}&=(2X_1+X_2-7)(X_1-1)\\
P_{1,0,1}&=(2X_1+X_2+X_3-11)(X_1-1),\\
P_{2,1}&=(3X_1X_2+2X_1X_3+X_2X_3-22X_1-11X_2+7X_3+59)(X_1-1),
\end{align*}
which gives for any $n$ great enough,
\begin{align*}
&\forall n\geq 3,&a_{n-3,0,1,0,\ldots,0}&=(-1)^n(n-1)!\sum_{1\leq p_1 \leq n-1} \frac{p_1-1}{p_1},\\
&\forall n\geq 4,&a_{n-4,2,0,\ldots,0}&=(-1)^n (n-1)!\sum_{1\leq p_1\leq n-1}\frac{(p_1-2)(p_1-1)}{p_1},\\
&\forall n\geq 5,&a_{n-5,1,1,0,\ldots,0}&=(-1)^{n+1}(n-1)!\sum_{1\leq p_1<p_2\leq n-1}\frac{(2p_1+p_2-7)(p_1-1)}{p_1p_2},\\
&\forall n\geq 6,&a_{n-6,1,0,1,0,\ldots,0}&=(-1)^n(n-1)!\sum_{1\leq p_1<p_2<p_3\leq n-1}\frac{(2p_1+p_2+p_3-11)(p_1-1)}{p_1p_2p_3}.
\end{align*}
This implies that all these sequences $(a_{n-2i_2-\ldots-ki_k,i_2,\ldots,i_k,0,\ldots,0})_n$ can be rewritten with the help of polynomial sequences and iterated harmonic sums
\[H^{(k)}=\left(
\sum_{1\leq p_1<\ldots<p_k \leq n}\frac{1}{p_1\ldots p_k}\right)_{n\geq 1}.\]
For example,
\begin{align*}
&\forall n\geq 3,&\frac{(-1)^n}{(n-1)!}a_{n-3,0,1,\ldots,0}&=n-1-H_{n-1}^{(1)},\\
&\forall n\geq 4,&\frac{(-1)^n}{(n-1)!}a_{n-4,2,0,0,\ldots,0}&=\frac{(n-1)(n-6)}{2}+2H_{n-1}^{(1)},\\
&\forall n\geq 5,&\frac{(-1)^{n+1}}{(n-1)!}a_{n-5,1,1,0,\ldots,0}&=(n-1)(n-10)+(10-n)H_{n-1}^{(1)}+7H_{n-1}^{(2)},\\
&\forall n\geq 6,&\frac{(-1)^n}{(n-1)!}a_{n-6,3,0,\ldots,0}&=\frac{(n^2-17n+90)(n-1)}{6}+(2n-14)H^{(1)}_{n-1}-8H^{(2)}_{n-1}.
\end{align*}

\textbf{Acknowledgments}. The author acknowledges support from the grant ANR-20-CE40-0007 \emph{Combinatoire Algébrique, Résurgence, Probabilités Libres et Opérades}.\\

\begin{notation}
The base field $\K$ is a commutative field of characteristic zero. All the spaces in this article will be taken over $\K$. 
\end{notation}

\section{Com-PreLie Hopf algebras}

\begin{defi}\label{defi1.1}
\begin{enumerate}
\item A \emph{(right) Com-PreLie algebra} \cite{Foissy28,Foissy30}  is a family $A=(A,m,\bullet)$, where $A$ is a vector space,
$m$ and $\bullet$ are bilinear products on $A$, such that
\begin{align*}
&\forall a,b\in A,&ab&=ba,\\
&\forall a,b,c\in A,&(a b) c&=a (b c),\\
&\forall a,b,c\in A,&(a\bullet b)\bullet c-a\bullet(b\bullet c)&=(a\bullet c)\bullet b-a\bullet(c\bullet b)&\mbox{(preLie identity)},\\
&\forall a,b,c\in A,&(a b)\bullet c&=(a\bullet c) b+a (b\bullet c)&\mbox{(Leibniz identity)}.
\end{align*}
In particular, $(A,m)$ is an associative, commutative algebra and $(A,\bullet)$ is a (right) preLie algebra.
We shall say that $A$ is unitary if the algebra $(A,m)$ is unitary.
\item A \emph{Com-PreLie bialgebra} is a family $(A,m,\bullet,\Delta)$, such that:
\begin{enumerate}
\item $(A,m,\bullet)$ is a unitary Com-PreLie algebra.
\item $(A,m,\Delta)$ is a bialgebra.
\item $\Delta$ and $\bullet$ are compatible, in the following sense:
\begin{align*}
&\forall a,b\in A,&\Delta(a\bullet b)&=a^{(1)}\otimes a^{(2)}\bullet b+a^{(1)}\bullet b^{(1)}\otimes a^{(2)} b^{(2)},
\end{align*}
with Sweedler's notation $\Delta(x)=x^{(1)}\otimes x^{(2)}$.
\end{enumerate}
If $(A,m,\Delta)$ is a Hopf algebra, of antipode denoted by $S$, we shall say that $(A,m,\bullet,\Delta)$ is a Com-PreLie Hopf algebra. 
\end{enumerate}\end{defi}

\begin{lemma}
Let $(A,m,\bullet,\Delta)$ be a Com-PreLie bialgebra, of counit $\varepsilon$ and of unit 1. For any $b\in A$,
$1\bullet b=0$. 
For any $a,b\in A$, $\varepsilon(a\bullet b)=0$.
\end{lemma}

\begin{proof}
For any $b\in A$,
\[1\bullet b=(1\cdot 1)\bullet b= (1\bullet b)\cdot 1+1\cdot(1\bullet b)=2 (1\bullet b),\]
so $1\bullet b=0$. Let $a,b\in A$.
\begin{align*}
\varepsilon(a\bullet b)&=(\varepsilon\otimes \varepsilon)\circ \Delta(a\bullet b)\\
&=\varepsilon\left(a^{(1)}\right)\varepsilon\left(a^{(2)}\bullet b\right)+\varepsilon\left(a^{(1)}\bullet b^{(1)}\right)\varepsilon\left(a^{(2)}\right)\varepsilon\left(b^{(2)}\right)\\
&=\varepsilon(a\bullet b)+\varepsilon(a\bullet b),
\end{align*}
so $\varepsilon(a\bullet b)=0$. 
\end{proof}

\begin{notation}
Let $(A,m,\Delta)$ be a bialgebra. Its augmentation ideal, that it to say the kernel of its counit $\varepsilon_A$, is denoted by $A_+$. It is given a coassociative, not necessarily counitary coproduct 
\[\tdelta:\left\{\begin{array}{rcl}
A_+&\longrightarrow&A_+\otimes A_+\\
a&\longmapsto&\Delta(a)-a\otimes 1_A-1_A\otimes a.
\end{array}
\right.\]
As it is coassociative, we can without any problem iterate this coproduct $\tdelta$: 
\begin{align*}
&\forall n\in \N,&\tdelta^{(n)}&=\begin{cases}
\id_{A_+}\mbox{ if }n=0,\\
\left(\tdelta \otimes \id_{A_+}^{\otimes (n-1)}\right)\circ \tdelta^{(n-1)}\mbox{ if }n\geq 1.
\end{cases}\end{align*}\end{notation}

We shall say that the bialgebra $A$ is conilpotent if for any $a\in A_+$, there exists $N(a)\geq 1$ such that $\tdelta^{(N(a))}(a)=0$. If this holds, then $(A,m,\Delta)$ is a Hopf algebra,
and its antipode is given by Takeuchi's formula \cite{Takeuchi1971}:
\begin{align*}
&\forall a\in A,&S(a)&=\sum_{n=0}^{N(a)-1} (-1)^{n+1}m_n\circ \tdelta^{(n-1)}(a),
\end{align*}
where for any $n\geq 0$,
\[m_n:\left\{\begin{array}{rcl}
A_+^{\otimes n}&\longrightarrow&A\\
a_1\otimes \ldots \otimes a_n&\longmapsto&a_1\cdot \ldots \cdot a_n.
\end{array}\right.\]

\begin{theo}\label{theo1.3}
Let $(A,m,\bullet,\Delta)$ be a conilpotent Com-PreLie Hopf algebra. Then for any $a,b \in A$,
\[S(a\bullet b)=\left(S(a)\bullet b^{(1)}\right)S(b^{(2)}).\]
\end{theo}

\begin{proof}
For any $b\in A$, as $1\bullet b=0$,
\[S(1\bullet b)=0=\left(S(1)\bullet b^{(1)}\right)S(b^{(2)}).\]
So the result hods if $a=1$: we can restrict ourselves to prove our result in the case where $a\in A_+=\ker(\varepsilon)$.

We put $A_n=\ker\left(\tdelta^{(n)}\right)$. As the coalgebra $(A,\Delta)$ is conilpotent, we obtain that $(A_n)_{n\geq 0}$ is an increasing filtration of $A_+$, with $A_0=(0)$. 
Consequently, for any $a\in A_+$, we put
\[\deg(a)=\min(n\in \N\mid a\in A_n).\]
We shall use Sweedler's notation $\tdelta(a)=a'\otimes a''$. Note that if $a\in A_n$, then 
\[\tdelta(a)\in A_{n-1}\otimes A_{n-1}.\]
Moreover, for any $a\in A_+$,
\begin{align*}
0&=\varepsilon(a)=m\circ (S\otimes \id)\circ \Delta(a)=S(a)+a+S(a')a'',
\end{align*}
so $S(a)=-a-S(a')a''$.\\

Let us prove that $S(a\bullet 1)=S(a)\bullet 1$ for any $a\in A_+$ by induction on $\deg(a)$. If $\deg(a)=0$, then $a=0$ and this is obvious. 
Let us assume the result at all ranks $<\deg(a)$. Then
\begin{align*}
\Delta(a\bullet 1)&=a\otimes 1\bullet 1+1\otimes a\bullet 1+a'\otimes a''\bullet 1+a\bullet 1\otimes 1+1\bullet 1\otimes a+a'\bullet 1\otimes a''\\
&=1\otimes a\bullet 1+a'\otimes a''\bullet 1+a\bullet 1\otimes 1+a'\bullet 1\otimes a'',
\end{align*}
so, with the induction hypothesis applied to $a'$,
\begin{align*}
0&=\varepsilon(a\bullet 1)\\
&=m\circ (S\otimes \id)\circ \Delta(a\bullet 1)\\
&=a\bullet 1+S(a')(a''\bullet 1)+S(a\bullet 1)+S(a'\bullet 1)a''\\
&=a\bullet 1+S(a')(a''\bullet 1)+(S(a')\bullet 1)a''+S(a\bullet 1)\\
&=a\bullet 1+(S(a')a'')\bullet 1+S(a\bullet 1)\\
&=(a+S(a')a'')\bullet 1+S(a\bullet 1)\\
&=-S(a)\bullet 1+S(a\bullet 1),
\end{align*}
so $S(a\bullet 1)=S(a)\bullet 1$. This concludes the induction. As a consequence, we can restrict ourselves to the case where $a,b\in A_+$. \\

Let us finally prove the result when $a,b\in A_+$ by induction on $N=\deg(a)+\deg(b)$. If $N\leq 1$, then $a=0$ or $b=0$ and the result is obvious. Let us assume the result at all ranks $<N$.
\begin{align*}
\Delta(a\bullet b)&=a\otimes 1\bullet b+1\otimes a\bullet b+a'\otimes a''\bullet b+a\bullet 1\otimes b+1\bullet 1\otimes b+a'\bullet 1\otimes a''b\\
&+a\bullet b\otimes 1+1\bullet b\otimes a+a'\bullet b\otimes a''+a\bullet b'\otimes b''+1\bullet b'\otimes ab''+a'\bullet b'\otimes a''b''\\
&=1\otimes a\bullet b+a\bullet b\otimes 1+a'\otimes a''\bullet b+a\bullet 1\otimes b+a'\bullet 1\otimes a''b+a'\bullet b\otimes a''\\
&+a\bullet b'\otimes b''+a'\bullet b'\otimes a''b''.
\end{align*}
Therefore, using the induction hypothesis on $a\bullet b$, $a'\bullet b'$ and $a\bullet b'$, with Sweedler's notation
\[(\tdelta \otimes \id)\circ \tdelta(b)=(\id \otimes \tdelta)\circ \tdelta(b)=b'\otimes b''\otimes b''',\]
\begin{align*}
0&=\varepsilon(a\bullet b)\\
&=m\circ(S\otimes \id)\circ \Delta(a\bullet b)\\
&=a\bullet b+S(a\bullet b)+S(a')(a''\bullet b)+(S(a)\bullet b')b''+(S(a)\bullet 1)S(b')b''+(S(a)\bullet b')S(b'')b'''\\
&+(S(a)\bullet 1)b+(S(a')\bullet b)a''+(S(a')\bullet 1)S(b)a''+(S(a')\bullet b')S(b'')a''+(S(a')\bullet b')a''b''\\
&+(S(a')\bullet 1)S(b')a''b''+(S(a')\bullet b')S(b'')a''b'''+(S(a')\bullet 1)a''b\\
&=a\bullet b+S(a\bullet b)+S(a')(a''\bullet b)\\
&+(S(a)\bullet b')(\underbrace{b''+S(b'')b'''}_{=-S(b'')})+(S(a)\bullet 1)(\underbrace{S(b')b''+b}_{=-S(b)})+(S(a')\bullet b)a''\\
&+(S(a')\bullet 1)a''(\underbrace{S(b)+S(b')b''+b}_{=0})+(S(a')\bullet b')a''(\underbrace{S(b'')+b''+S(b'')b'''}_{=0})\\
&=a\bullet b+S(a\bullet b)+S(a')(a''\bullet b)+(S(a')\bullet b)a''-(S(a)\bullet b')S(b'')-(S(a)\bullet 1)S(b)\\
&=a\bullet b+S(a\bullet b)+(S(a')a'')\bullet b-(S(a)\bullet b')S(b'')-(S(a)\bullet 1)S(b)\\
&=(\underbrace{a+S(a')a''}_{=-S(a)})\bullet b+S(a\bullet b)-(S(a)\bullet b')S(b'')-(S(a)\bullet 1)S(b)\\
&=S(a\bullet b)-(S(a)\bullet b^{(1)})S(b^{(2)}),
\end{align*}
so the result holds for $a\bullet b$. 
\end{proof}

\section{Application to the Connes-Moscovici subalgebra}

The Connes-Kreimer Hopf $H_{CK}$ algebra is given a preLie product by graftings \cite{Foissy56}. Let us recall first a few definitions.
If $F$ and $G$ are two rooted forests and $v\in V(F)$, then $F\bullet_v G$ is the rooted forest obtained from $FG$ by adding an edge from $v$ to any root of $G$.
We then define the preLie product $\bullet$ by
\begin{align*}
&\forall F,G\in \F,&F\bullet G&=\sum_{v\in V(F)} F\bullet_v G.
\end{align*}
For example, 
\begin{align*}
\tun \bullet \tun\tun\tun&=\tquatreun,&\tun\tun\tun\bullet \tun&=\tdeux\tun\tun+\tun\tdeux\tun+\tun\tun\tdeux=3\tdeux\tun \tun,\\
\tun\bullet \tdeux\tun&=\tquatredeux,&\tdeux\tun\bullet \tun&=\ttroisun\tun+\ttroisdeux\tun+\tdeux\tdeux,\\
\tun\bullet \ttroisun&=\tquatrequatre,&\ttroisun\bullet \tun&=\tquatreun+\tquatredeux+\tquatretrois=\tquatreun+2\tquatredeux,\\
\tun\bullet \ttroisdeux&=\tquatrecinq,&\ttroisdeux\bullet \tun&=\tquatredeux+\tquatrequatre+\tquatrecinq,\\
\tdeux\bullet \tdeux&=\tquatredeux+\tquatrecinq.
\end{align*}
In particular, the grafting operator, for which $H_{CK}$ satisfies a universal property \cite{Connes98}, is
\[B:\left\{\begin{array}{rcl}
H_{CK}&\longrightarrow&H_{CK}\\
x&\longmapsto&\tun\bullet x.
\end{array}\right.\]
It associates to any rooted forest $F$ the tree obtained by adding a new vertex and graft all the roots of $F$ on it.
The growth operator used in \cite{ConnesMoscovici98} is
\[N:\left\{\begin{array}{rcl}
H_{CK}&\longrightarrow&H_{CK}\\
x&\longmapsto&x\bullet \tun.
\end{array}\right.\]
It associates to any rooted forest $F$ the sum of forests obtained by adding a new leaf to $F$, grafted on a vertex of $F$. 
The Connes-Moscovici Hopf subalgebra $H_{CM}$ is generated by the inductively defined generators $\delta_n$:
\begin{align*}
&\forall n\geq 1,&\delta_n&=\begin{cases}
\tun \mbox{ if }n=1,\\
\delta_{n-1}\bullet \tun\mbox{ if }n\geq 2.
\end{cases}
\end{align*}
Our aim here is to study the antipode of $S(\delta_n)$. Firstly, as $H_{CK}$ is graded by the number of vertices of the forests, and as the only forest of degree $0$ is $1$,
it is conilpotent as a coalgebra and we can apply Theorem \ref{theo1.3}. Secondly, as $H_{CM}$ is a Hopf subalgebra of $H_{CK}$, and by the homogeneity of the antipode, 
we can write for any $n \geq 1$, 
\begin{align*}
S(\delta_n)&=\sum_{\substack{(i_1,\ldots,i_n) \in \N^n,\\1i_1+\ldots+ni_n=n}}a_{i_1,\ldots,i_n}\delta_1^{i_1}\ldots \delta_n^{i_n}.
\end{align*}

\begin{example}
Direct computations give
\begin{align*}
S(\delta_1)&=-\delta_1,&
S(\delta_2)&=-\delta_2+\delta_1^2,&
S(\delta_3)&=-\delta_3+4\delta_1\delta_2-2\delta_1^3,
\end{align*}
so 
\begin{align*}
a_1&=-1,&a_{0,1}&=-1,&a_{0,0,1}&=-1,\\
&&a_{2,0}&=1,&a_{1,1,0}&=4,\\
&&&&a_{3,0,0}&=-2.
\end{align*}
\end{example}

Let us give an inductive way to compute these coefficients.

\begin{prop}\label{prop2.1}
Let $(i_1,\ldots,i_n)\in \N^n$, such that $1i_1+\ldots+ni_n=n$. If $i_n=1$ (and, therefore, $i_1=\ldots=i_{n-1}=0$), then $a_{i_1,\ldots,i_n}=-1$. Otherwise, $i_n=0$ and
\[a_{i_1,\ldots,i_n}=\sum_{j\geq 2,\: i_j \geq 1}(i_{j-1}+1)a_{i_1,\ldots,i_{j-1}+1,i_j-1,\ldots,i_{n-1}}
-(n-1)\delta_{i_1\geq 1}a_{i_1-1,i_2,\ldots,i_{n-1}},\]
where the symbol $\delta_{i_1\geq 1}$ takes the value 1 if $i_1\geq 1$ and $0$ if $i_1=0$.
\end{prop}

\begin{proof}
We consider the two operators $N,D:H_{CM}\longrightarrow H_{CM}$, defined by
\begin{align*}
N(x)&=x\bullet \tun,&D(x)&=x\bullet 1.
\end{align*}
Note that if $x$ is homogeneous of a certain degree $k$, then $D(x)=kx$. 
By Theorem \ref{theo1.3}, for any $x\in H_{CM}$,
\[S\circ N(x)=S(x)\bullet \tun+(S(x)\bullet 1)S(\tun)=N\circ S(x)-\tun D\circ S(x).\]
In particular, if $x$ is homogeneous of a certain degree $k$, as $S(x)$ is homogeneous of the same degree,
\[S\circ N(x)=N\circ S(x)-k\tun S(x).\]
Applying this to $x=\delta_{n-1}$, homogeneous of degree $k=n-1$,
\begin{align*}
S(\delta_n)&=\sum_{\substack{(i_1,\ldots,i_{n-1}) \in \N^{n-1},\\1i_1+\ldots+(n-1)i_{n-1}=n-1}}a_{i_1,\ldots,i_{n-1}}N(\delta_1^{i_1}\ldots \delta_{n-1}^{i_{n-1}})\\
&-(n-1)\sum_{\substack{(i_1,\ldots,i_{n-1}) \in \N^{n-1},\\1i_1+\ldots+(n-1)i_{n-1}=n-1}}a_{i_1,\ldots,i_{n-1}}\delta_1^{i_1+1}\delta_2^{i_2}\ldots \delta_{n-1}^{i_{n-1}}\\
&=\sum_{\substack{(i_1,\ldots,i_{n-1}) \in \N^{n-1},\\1i_1+\ldots+(n-1)i_{n-1}=n-1}}\sum_{j=1}^{n-1}i_j a_{i_1,\ldots,i_{n-1}}\delta_1^{i_1}\ldots \delta_j^{i_j-1} \delta_{j+1}^{i_j+1}\delta_{n-1}^{i_{n-1}}\\
&-(n-1)\sum_{\substack{(i_1,\ldots,i_{n-1}) \in \N^{n-1},\\1i_1+\ldots+(n-1)i_{n-1}=n-1}}a_{i_1,\ldots,i_{n-1}}\delta_1^{i_1+1}\delta_2^{i_2}\ldots \delta_{n-1}^{i_{n-1}}.
\end{align*}
The result follows by an identification of the coefficients.
\end{proof}

Here are a few examples, some of them being referenced in the OEIS \cite{Sloane}. 

\begin{lemma} \label{lem2.2}
\begin{align*}
&\forall n\geq 1,&a_{n,0,\ldots,0}&=(-1)^n(n-1)!,&&\mbox{\rm OEIS A000142}\\
&\forall n\geq 2,&a_{n-2,1,0,\ldots,0}&=(-1)^{n-1}(n-1)!(n-1),&&\mbox{\rm OEIS A001563}\\
&\forall n\geq 3,&a_{1,0,\ldots,0,1,0}&=\frac{n(n-1)}{2}+1,&&\mbox{\rm OEIS A152947}\\
&\forall n\geq 3,&a_{n-3,0,1,0,\ldots,0}&=(-1)^n(n-1)!\sum_{1\leq p_1 \leq n-1} \frac{p_1-1}{p_1},&&\mbox{\rm OEIS A067318}\\
&\forall n\geq 4,&a_{n-4,2,0,\ldots,0}&=(-1)^n (n-1)!\sum_{1\leq p_1\leq n-1}\frac{(p_1-2)(p_1-1)}{p_1}.
\end{align*}
\end{lemma}

\begin{proof}
By Proposition \ref{prop2.1}, 
\begin{align*}
&\forall n\geq 2,&a_{n,0,\ldots,0}&=-(n-1)a_{n-1,0,\ldots,0},\\
&\forall n\geq 3,&a_{n-2,1,0,\ldots,0}&=(n-1)a_{n-1,0,\ldots,0}-(n-1)a_{n-3,1,0,\ldots,0}\\
&&&=(-1)^{n-1}(n-1)!-(n-1)a_{n-3,1,0,\ldots,0},\\
&\forall n\geq 3,&a_{\underbrace{\mbox{\scriptsize 1,0,\ldots,0,1,0}}_{n}}&=a_{\underbrace{\mbox{\scriptsize  1,0,\ldots,0,1,0}}_{n-1}}-
(n-1)a_{0,\ldots,0,1}a_{\underbrace{\mbox{\scriptsize 1,0,\ldots,0,1,0}}_{n-1}}+n-1,\\
&\forall n\geq 4,&a_{n-3,0,1,0,\ldots,0}&=a_{n-3,1,0,\ldots,0}(n-1)a_{n-4,0,1,0,\ldots,0}\\
&&&=(-1)^n(n-2)!(n-2)-(n-1)a_{n-4,0,1,0,\ldots,0}.
\end{align*}
Moreover,
\begin{align*}
a_{1}&=-1,&a_{0,1}&=-1,&a_{1,1,0}&=4,&a_{0,0,1}&=-1.
\end{align*}
The four first results follow from an easy induction.
For $a_{n-4,2,0,\ldots,0}$, we proceed by induction on $n$. For $n=0$, still by Proposition \ref{prop2.1}, 
\[a_{0,2,0,0}=a_{1,1,0}=4\]
whereas
\[3!\sum_{1\leq p_1\leq 3}\frac{(p_1-1)(p_1-2)}{p_1}=3!\frac{2}{3}=4.\]
Let us assume the result at rank $n-1$, with $n\geq 5$. 
\begin{align*}
a_{n-4,2,0,\ldots,0}&=(n-3)a_{n-3,1,0,\ldots,0}-(n-1)a_{n-5,2,0,\ldots,0}\\
&=(n-3)(-1)^n(n-2)!(n-2)-(-1)^{n-1}(n-1)(n-2)!\sum_{1\leq p_1\leq n-2}\frac{(p_1-1)(p_1-2)}{p_1}\\
&=(-1)^n(n-1)! \frac{(n-2)(n-3)}{n-1+}(-1)^n(n-1)(n-2)!\sum_{1\leq p_1\leq n-2}\frac{(p_1-1)(p_1-2)}{p_1}\\
&=(-1)^n(n-1)!\sum_{1\leq p_1\leq n-1}\frac{(p_1-1)(p_1-2)}{p_1}. \qedhere
\end{align*}\end{proof}

Let us now describe these coefficients as iterated sums. We shall need the following notions:

\begin{defi}
A dominant sequence of length $k\geq 1$ is a sequence $(i_1,\ldots,i_k)\in \N^k$ such that $i_k\neq 0$.
If $(i_1,\ldots,i_n)\in \N^n$ is nonzero, the dominant of $(i_1,\ldots,i_n)$ is the dominant sequence
\begin{align*}
\partial(i_1,\ldots,i_n)&=(i_1,\ldots,i_m)&\mbox{where}&&m&=\max\{j\mid i_j\neq 0\}
\end{align*}
By convention, $\partial(0,\ldots,0)$ is the empty sequence, denoted by $()$.
\end{defi}

\begin{defi}
Let $(i_2,\ldots,i_n)\in \N^{n-1}$. Its weight is $\omega(i_2,\ldots,i_n)=1i_2+\ldots+(n-1)i_n-1$.
\end{defi}

\begin{remark}
Obviously, $\omega \circ \partial=\omega$.
\end{remark}

We finally introduce a family of polynomials, indexed by dominant sequences of weight $\geq 1$:

\begin{defi}
We define $P_{i_2,\ldots,i_n}\in \mathbb{Z}[p_1,\ldots,X_{\omega(i_2,\ldots,i_n)}]$ for any dominant sequence $(i_2,\ldots,i_n)$ of weight $\geq 1$ by the following:
\begin{align*}
P_2&=(X_1-2)(X_1-1),\\
P_{0,1}&=X_1-1,\\
P_{i_2,\ldots,i_n}&=\sum_{j\geq 3,i_j\geq 1}(i_{j-1}+1)P_{\partial(i_2,\ldots,i_{j-1}+1,i_j-1,\ldots,i_n)}\\
&+\delta_{i_2\geq 1}(X_{\omega(i_2,\ldots,i_n)}-2i_2-\ldots -ni_n+2)P_{\partial(i_2-1,i_3,\ldots,i_n)}&\mbox{if }\omega(i_2,\ldots,i_n)\geq 3,
\end{align*}
where
\[\delta_{i_2\geq 1}=\begin{cases}
1\mbox{ if }i_2\geq 1,\\
0\mbox{ otherwise}.
\end{cases}\]
\end{defi}

\begin{remark}
If $j\geq 3$ and $i_j\geq 1$,
\[\omega(i_2,\ldots,i_{j-1}+1,i_j-1,\ldots,i_n)=\omega(i_2,\ldots,i_n)-1\]
and if $i_2\geq 1$,
\[\omega(i_2-1,i_3,\ldots,i_n)=\omega(i_2,\ldots,i_n)-1.\]
Therefore, $P_{i_2,\ldots,i_n}$ is well-defined for any $(i_2,\ldots,i_n)$ of weight $\geq 2$.
\end{remark}

\begin{remark}
An easy induction proves that for any dominant sequence $(i_2,\ldots,i_n)$ of weight $\geq 1$, 
\[\deg(P_{i_2,\ldots,i_n})=i_2+\ldots+i_n,\]
and that $P_{i_2,\ldots,i_n}$ is a multiple of $X_1-1$. 
\end{remark}

\begin{lemma}\label{lem2.6}
Let $(i_2,\dots,i_n)$ be a dominant sequence of weight $N\geq 2$. If $p_1,\ldots,p_N\in \N$, such that
\[1\leq p_1<\ldots<p_N\leq 2i_2+\ldots+ni_n-2,\]
then $P_{i_2,\ldots,i_n}(p_1,\ldots,p_N)=0$.
\end{lemma}

\begin{proof}
By induction on $N$, then $(i_2,\ldots,i_n)=(2)$ or $(0,1)$. In the first case, $2i_2+\ldots+ni_n-2=2$ and as $P_2=(X_1-1)(X_2-1)$, the result is obtained. 
In the second case, $2i_2+\ldots+ni_n-2=1$ and as $P_{0,1}=X_1-1$, the result is obtained. Let us assume the result at a rank $N-1$, with $N\geq 3$.
Let us consider $(p_1,\ldots,p_N)$ such that $1\leq p_1<\ldots<p_N\leq 2i_2+\ldots+ni_n-2$. If $j\geq 3$ and $i_j\geq 1$, then 
\[1\leq p_1<\ldots<p_{n-1}\leq 2i_2+\ldots+ni_n-3=2i_2+\ldots+(j-1)(i_{j-1}+1)+j(i_j+1)+\ldots+ni_n-2.\]
By the induction hypothesis, $P_{\pi(i_2,\ldots,i_{j-1}+1,i_j+1,\ldots,i_n)}(p_1,\ldots,p_{N-1})=0$. Let us assume that $i_2\geq 1$. If $p_N\leq 2i_2+\ldots+ni_n-3$, then 
\[1\leq p_1<\ldots<p_{N-1}\leq (i_2-1)+3i_3+\ldots+ni_n-2.\]
By the induction hypothesis, $P_{\partial(i_2-1,\ldots,i_n)}(p_1,\ldots,p_{N-1})=0$. Otherwise, $p_N=2i_2+\ldots+ni_n-2$ and $p_{\omega(i_2,\ldots,i_n)}-2i_2-\ldots -ni_n+2=0$.
Combining all these, we obtain $P_{i_2,\ldots,i_n}(p_1,\ldots,p_N)=0$.
\end{proof}

\begin{theo}
Let $(i_1,\ldots,i_n)\in \N^n$ such that $1i_1+\ldots+ni_n=n$ and $N=\omega(i_2,\ldots,i_n)\geq 1$. Then
\[a_{i_1,\ldots,i_n}=(-1)^{i_1+\ldots+i_n}(n-1)!\sum_{1\leq p_1<\ldots<p_N\leq n-1}
\frac{P_{\partial(i_2,\ldots,i_n)}(p_1,\ldots,p_N)}{p_1\ldots p_N}.\]
\end{theo}

\begin{proof}
In order to simplify the proof, we put 
\[b_{i_1,\ldots,i_n}=\frac{(-1)^{i_1+\ldots+i_n}}{(n-1)!}a_{i_1,\ldots,i_n}.\]
These coefficients satisfy the induction
\[b_{i_1,\ldots,i_n}=\sum_{j\geq 2,\: i_j \geq 1}(i_{j-1}+1)\frac{1}{n-1}b_{i_1,\ldots,i_{j-1}+1,i_j-1,\ldots,i_{n-1}}
+\delta_{i_1\geq 1}b_{i_1-1,i_2,\ldots,i_{n-1}}.\]
Our aim is no to prove that
\[b_{i_1,\ldots,i_n}=\sum_{1\leq p_1<\ldots<p_N\leq n-1}\frac{P_{\partial(i_2,\ldots,i_n)}(p_1,\ldots,p_N)}{p_1\ldots p_N}.\]
We proceed by induction on the weight $N=\omega(i_2,\ldots,i_n)$. If $N=2$, two cases occur,
If $(i_2,\ldots,i_n)=(2)$ or $(0,1)$. These are covered by Lemma \ref{lem2.2}.
Let us now assume the result for all sequences of weight $N-1$, with $N\geq 3$. We proceed by induction on $i_1=n-2i_2-\ldots-ni_n$. By Lemma \ref{lem2.6},
\begin{align*}
\sum_{1\leq p_1<\ldots<p_N\leq n-1} \frac{P_{\partial(i_2,\ldots,i_n)}(p_1,\ldots,p_N)}{p_1\ldots p_N}
&=\sum_{1\leq p_1<\ldots<p_{N-1}\leq n-2} \frac{P_{\partial(i_2,\ldots,i_n)}(p_1,\ldots,p_{N-1},n-1)}{p_1\ldots p_{N-1}(n-1)}.
\end{align*}
Moreover,
\begin{align*}
b_{0,i_2,\ldots,i_n}&=\sum_{j\geq 2,\: i_j\geq 1}\frac{i_{j-1}+1}{n-1}\sum_{1\leq p_1<\ldots<p_{N-1}\leq n-2} \frac{P_{i_2,\ldots,i_{j-1}+1,i_j-1,\ldots,i_N}(p_1,\ldots,p_{N-1})}{p_1\ldots p_{N-1}}\\
&=\sum_{j\geq 3,\: i_j\geq 1}\frac{i_{j-1}+1}{n-1}\sum_{1\leq p_1<\ldots<p_{N-1}\leq n-2} \frac{P_{i_2,\ldots,i_{j-1}+1,i_j-1,\ldots,i_N}(p_1,\ldots,p_{N-1})}{p_1\ldots p_{N-1}}\\
&+\delta_{i_2\geq 1}\frac{n-1-2i_2-\ldots-ni_n+2}{n-1}\sum_{1\leq p_1<\ldots<p_{N-1}\leq n-2} \frac{P_{i_2-1,\ldots,i_N}(p_1,\ldots,p_{N-1})}{p_1\ldots p_{N-1}}\\
&=\sum_{1\leq p_1<\ldots<p_{N-1}\leq n-2} \frac{P_{\partial(i_2,\ldots,i_n)}(p_1,\ldots,p_{N-1},n-1)}{p_1\ldots p_{N-1}(n-1)}\\
&=\sum_{1\leq p_1<\ldots<p_N\leq n-1} \frac{P_{\partial(i_2,\ldots,i_n)}(p_1,\ldots,p_N)}{p_1\ldots p_N}.
\end{align*}
Let us assume the result at rank $i_1-1$. 
\begin{align*}
b_{i_1,i_2,\ldots,i_n}&=\sum_{j\geq 2,\: i_j\geq 1}\frac{i_{j-1}+1}{n-1}\sum_{1\leq p_1<\ldots<p_{N-1}\leq n-2} \frac{P_{i_2,\ldots,i_{j-1}+1,i_j-1,\ldots,i_N}(p_1,\ldots,p_{N-1})}{p_1\ldots p_{N-1}}+b_{i_1-1,i_2,\ldots,i_n}\\
&=\sum_{1\leq p_1<\ldots<p_{N-1}\leq n-2} \frac{P_{\partial(i_2,\ldots,i_n)}(p_1,\ldots,p_{N-1},n-1)}{p_1\ldots p_{N-1}(n-1)}\\
&+\sum_{1\leq p_1<\ldots<p_N\leq n-2} \frac{P_{\partial(i_2,\ldots,i_n)}(p_1,\ldots,p_N)}{p_1\ldots p_N}\\
&=\sum_{1\leq p_1<\ldots<p_N\leq n-1} \frac{P_{\partial(i_2,\ldots,i_n)}(p_1,\ldots,p_N)}{p_1\ldots p_N}.
\end{align*}
So the property is true at all rank $N\geq 2$.
\end{proof}

\begin{remark}
\begin{enumerate}
\item If $(i_2,\ldots,i_n)=()$, then $a_{n,0,\ldots,0}=(-1)^n (n-1)$. We can put $\omega(())=0$ and $P_{()}=1$. The result then holds also in this case, by convention.
\item If $(i_2,\ldots,i_n)=(1)$, then $a_{n-2,1,0,\ldots,0}=(-1)^{n-1}(n-1)!(n-1)$. The result holds if we take $\omega(1)=1$ and $P_{1}=p_1$. 
\end{enumerate}
\end{remark}

\begin{example}
Easy inductions prove that
\begin{align*}
P_{0,\ldots,0,1}&=X_1-1,&P_n&=(X_1-1)(X_1-1)(X_2-4)\ldots (X_{n-1}-2n+2).
\end{align*}
Here are other examples:
\begin{align*}
P_{1,1}&=(2X_1+X_2-7)(X_1-1)\\
P_{1,0,1}&=(2X_1+X_2+X_3-11)(X_1-1),\\
P_{2,1}&=(3X_1X_2+2X_1X_3+X_2X_3-22X_1-11X_2+7X_3+59)(X_1-1),\\
P_{1,0,0,1}&=(2X_1+X_2+X_3+X_4-16)(X_1-1),\\
P_{0,1,1}&=(6X_1+3X_2+X_3-25)(X_1-1),\\
P_{2,0,1}&=(3X_1X_2+2X_1X_3+X_2X_3+2X_1X_4+X_2X_4+X_3X_4-34X_1\\
&-17X_2-13X_3-11X_4+125)(X_1-1),\\
P_{1,2}&=(6X_1X_2+4X_1X_3+2X_2X_3+2X_1X_4+X_2X_4-56X_1-28X_2-14X_3\\
&-7X_4+160)(X_1-1),\\
P_{3,1,0,0,0}&=(4X_1X_2X_3+3X_1X_2X_4+2X_1X_3X_4+X_2X_3X_4-45X_1X_2-30X_1X_3-15X_2X_3\\
&-22X_1X_4-11X_2X_4-7X_3X_4+250X_1+125X_2+81X_3+59X_4-605)(X_1-1).
\end{align*}
\end{example}

\begin{example}\begin{enumerate}
\item In particular, for any $n\geq 4$,
\[a_{n-4,0,0,1,0,\ldots,0}=-(n-1)!\sum_{2\leq i_1<i_2\leq n-1} \frac{i_1-1}{i_1i_2}.\]
We recover Entry A122105 of the OEIS \cite{Sloane}. 
\item We also obtain
\begin{align*}
a_{1,0,\dots,0,1,0}&=(n-1)!\sum_{2\leq i_1<\ldots<i_{n-3}\leq n-1}\frac{i_1-1}{i_1\ldots i_{n-3}}\\
&=(n-1)!\left(\frac{4}{(n-1)!}+\sum_{j=3}^{n-1} \frac{j}{(n-1)!}\right)\\
&=4+\sum_{j=3}^{n-1}j\\
&=\frac{n(n-1)}{2}+1.
\end{align*}
We recover the formula of Lemma \ref{lem2.2}. Here, $j$ corresponds to the missing element of $\{2,\ldots,n-1\}$ in $\{i_2,\ldots,i_{n-3}\}$,
the case $j=2$ (the unique case where $i_1\neq 2$) being treated separately. 
\end{enumerate}\end{example}

\begin{notation}
We put, for any $k,n\geq 1$, 
\begin{align*}
H_n^{(k)}&=\sum_{1\leq p_1<\ldots <p_k\leq n}\frac{1}{p_1\ldots p_k}.
\end{align*}
\end{notation}

For a fixed dominant sequence $(i_2,\ldots,i_n)$, using technical manipulations of sums, the sequence $(b_{n-2i_2-\ldots-ni_n,i_2,\ldots,i_n,0,\ldots,0})_{n\geq 2i_2+\ldots+n_in}$
is the sum of a polynomial sequence (with coefficients in $\mathbb{Q}$) and a $\mathbb{Q}$-linear span of sequences $(n^kH_{n-1}^{(l)})$, with $k\geq 0$ and $l\geq 1$.

\begin{example} We obtain
\begin{align*}
&\forall n\geq 3,&b_{n-3,0,1,\ldots,0}&=\sum_{1\leq p_1\leq n-1}\frac{p_1-1}{p_1}\\
&&&=n-1-H_{n-1}^{(1)},\\
&\forall n\geq 4,&b_{n-4,2,0,0,\ldots,0}&=\sum_{1\leq p_1\leq n-1}\frac{(p_1-1)(p_1-2)}{p_1}\\
&&&=\frac{(n-1)(n-6)}{2}+2H_{n-1}^{(1)},\\
&\forall n\geq 5,&b_{n-5,1,1,0,\ldots,0}&=\sum_{1\leq p_1<p_2\leq n-1}\frac{(2p_1+p_2-7)(p_1-1)}{p_1p_2}\\
&&&=(n-1)(n-10)+(10-n)H_{n-1}^{(1)}+7H_{n-1}^{(2)},\\
&\forall n\geq 6,&b_{n-6,3,0,\ldots,0}&=\frac{(n^2-17n+90)(n-1)}{6}+(2n-14)H^{(1)}_{n-1}-8H^{(2)}_{n-1},\\
&\forall n\geq 7,&b_{n-7,2,1,0,\ldots,0}&=\frac{(n^2-25n+210)(n-1)}{2}\\
&&&+\frac{-n^2+29n-109}{2}H_{n-1}^{(1)}+(7n-99)H_{n-1}^{(2)}-59H_{n-1}^{(3)},\\
&\forall n\geq 8,&b_{n-8,4,0,\ldots,0}&=\frac{(n^3-33n^2+434n-2520)(n - 1)}{24}\\
&&&+(n^2-19n+108)H_{n-1}^{(1)}+(-8n+92)H_{n-1}^{(2)}+48H_{n-1}^{(3)}.
\end{align*} \end{example}

\begin{prop}
Let $k\geq 2$. Then
\[b_{n-k,0,\ldots,0,1,0,\ldots,0}=n-1-\sum_{i=0}^{k-2}H_{n-1}^{(i)}.\]
Note that the 1 is in position $k$ in $(n-k,0,\ldots,0,1,0,\ldots,0)$.
\end{prop}
\begin{proof}
By induction on $k$. If $k=2$, by Lemma \ref{lem2.2}, $b_{n-2,2,0,\ldots,0}=n-1$. Let us assume the result at rank $k-1$. As $P_{0,\ldots,0,1}=X_1-1$, 
\begin{align*}
b_{n-k,0,\ldots,0,1,0,\ldots,0}&=\sum_{1\leq p_1<\ldots<p_{k-2}\leq n-1} \frac{p_1-1}{p_1\ldots p_{k-2}}\\
&=-H_{n-1}^{(k-2)}+\sum_{1\leq p_1<\ldots<p_{k-2}\leq n-1} \frac{1}{p_2\ldots p_{k-2}}\\
&=-H_{n-1}^{(k-2)}+\sum_{1\leq p_2\ldots<p_{k-2}\leq n-1} \frac{p_2-1}{p_2\ldots p_{k-2}}\\
&=-H_{n-1}^{(k-2)}+b_{n-k+1,0,\ldots,1,0,\ldots,0}
\end{align*}
and the result follows. Note that in the last line of this computation, the 1 is in position $k-1$ in $(n-k+1,0,\ldots,1,0,\ldots,0)$.
\end{proof}

\bibliographystyle{amsplain}
\bibliography{biblio}

\end{document}